\documentclass[12pt,letterpaper]{amsart}
\makeatletter
\@namedef{subjclassname@2020}{%
\textup{2020} Mathematics Subject Classification}
\makeatother
\usepackage{geometry}
\usepackage{amssymb}
\usepackage{mathpazo} 
\usepackage{setspace}
\usepackage[colorlinks,allcolors=blue]{hyperref} 
\usepackage{xpatch} 
\xpatchcmd{\proof}{\itshape}{\bfseries}{}{}
\newtheorem{theorem}{Theorem}
\newtheorem{theoremA}{Theorem}

\newtheorem{lemma}{Lemma}
\newtheorem{corollary}{Corollary}
\theoremstyle{remark}
\newtheorem{remark}{Remark}
\newtheorem{example}{Example}
\newcommand{\C}{\mathbb{C}}
\newcommand{\D}{\Omega}
\newcommand{\ep}{\varepsilon}

\newcommand{\zb}{\overline{z}}
\newcommand{\wb}{\overline{w}}
\newcommand{\Balpha}{\boldsymbol{\alpha}}
\newcommand{\Bbeta}{\boldsymbol{\beta}}
\newcommand{\Btheta}{\boldsymbol{\theta}}
\onehalfspace

\title{On spectra of Hankel operators on the polydisc}

\author{\v{Z}eljko \v{C}u\v{c}kovi\'c} 
\address[\v{Z}eljko \v{C}u\v{c}kovi\'c]{University of Toledo, Department of 
				Mathematics \& Statistics, Toledo, OH 43606, USA} 
 \email{Zeljko.Cuckovic@utoledo.edu} 

\author{Zhenghui Huo}
\address[Zhenghui Huo]{Duke Kunshan University, Division of Natural 
	and Applied Sciences, Kunshan, Jiangsu 215316, China}
\email{zhenghui.huo@dukekunshan.edu.cn} 

\author{S\"{o}nmez \c{S}ahuto\u{g}lu} 
\address[S\"{o}nmez \c{S}ahuto\u{g}lu]{University of Toledo, Department of 	
				Mathematics \& Statistics, Toledo, OH 43606, USA} 
\email{Sonmez.Sahutoglu@utoledo.edu} 

\date{\today}
\begin{document}

\begin{abstract}
We give sufficient conditions for the essential spectrum of the Hermitian square of a class of 
Hankel operators on the Bergman space of the polydisc to contain intervals.  We also 
compute the spectrum in case the symbol is a monomial.
\end{abstract}

\maketitle

The study of spectral properties of Toeplitz and Hankel operators acting on the Bergman 
space is a difficult topic.  In recent years some progress has been made in understanding 
spectral properties of Toeplitz operators.  We highlight the result by Sundberg and Zheng 
\cite{SundbergZheng10}  who have proved that the spectra and essential spectra of Toeplitz 
operators on the unit disc $\mathbb{D}$ need not be connected. Their result shows a sharp 
difference with the spectra of Toeplitz operators on the Hardy space.  As is well-known, 
Widom \cite{Widom64} showed that for any $L^{\infty}$  symbol $\psi$ on the unit circle, 
the spectrum of the Toeplitz operator $T_{\psi}$ is  a connected subset of $\C$. Similarly 
Douglas \cite[Theorem 7.45]{DouglasBook} proved that the essential spectrum of $T_{\psi}$ 
is also connected. In this context, we also mention papers \cite{ZhaoZheng16,GuoZhaoZheng2023} 
which study spectra of certain classes of the Bergman Toeplitz operators on the unit disk.  
On the other hand,  we are not aware of much work done about the spectral 
properties of Hankel  operators.  

Since the Hankel operator $H_{\psi}$ does not map the Bergman space into itself, we will 
consider the Hermitian square $H_{\psi}^* H_{\psi}$ and we will obtain some initial results 
about the spectrum. In this paper we will only be concerned with symbols that are continuous 
up to the boundary of the domain.  In the case of the unit disc, $H_{\psi}$ is compact for 
$\psi\in C(\overline{\mathbb{D}})$ (see, for instance, \cite[Proposition 8]{AxlerConwayMcDonald82}); 
hence the spectrum of  $H_{\psi}^* H_{\psi}$ is discrete. Like the situation on the unit disc, 
on bounded strongly pseudoconvex domains in $\C^n$, $H_{\psi}$ is compact  \cite{CuckovicSahutoglu09}.  
Then on such domains the spectrum of $H_{\psi}^* H_{\psi}$ is discrete as well.  
In this paper, we focus on the polydisc  and find some sufficient conditions in terms of behavior 
$\psi$ on the boundary so that the spectrum of $H_{\psi}^* H_{\psi}$ contains intervals. 
One of the reasons for this departure from the one-dimensional and strongly pseudoconvex 
case is the fact that Hankel operators with symbols continuous on $\overline{\mathbb{D}^2}$ 
may not be compact (see \cite{Le10,ClosSahutoglu18}). We note that, in case the symbol is 
smooth on the closure the same result was proven in  \cite{CuckovicSahutoglu09}. With 
regard to Sundberg-Zheng result, from the same papers mentioned in this paragraph, 
we know that if the symbol $\psi$ is holomorphic along any disc in the boundary of a convex 
domain, then $H_{\psi}$ is compact \cite{CuckovicSahutoglu09,CelikSahutogluStraube23,Zimmer2023}; 
hence the spectrum of $H_{\psi}^* H_{\psi}$ is disconnected. 

At this point we would like to mention the well known formula connecting Hankel and 
Toeplitz operators 
\[H_{\psi}^* H_{\psi} = T_{|\psi|^2} - T_{\overline{\psi}} T_{\psi}.\]

Thus our results could shed a new light on the spectra of semicommutators of Toeplitz operators. 
For more information about Hankel and Toeplitz operators on the unit disc we refer the reader to 
a standard reference \cite{ZhuBook}. 

In the rest of the paper, we prove two sufficient conditions for the spectrum of 
$H^*_{\psi} H_{\psi}$ to contain intervals. We also compute the spectrum of 
$H^*_{\psi} H_{\psi}$ in the case $\psi$ is monomial.

\section{Main Results}
Let $\D$ be a bounded domain in $\C^n$ and $A^2(\D)$ denote the Bergman space, 
the set of square integrable holomorphic functions on $\D$.  We denote the Bergman projection 
by $P^{\D}$. Then for a bounded measurable function $\psi$ on $\Omega$, Hankel operator 
$H^{\Omega}_\psi$  on $A^2(\D)$ with symbol $\psi$ is defined as 
\[H^{\Omega}_{\psi} f=(I-P^{\D})(\psi f)\]
for  $f\in A^2(\Omega).$
Here $I$ denotes the identity operator. For simplicity, we will simply write $H_\psi$ when there is 
no confusion about the domain. We note that the Toeplitz operator $T_{\psi}$ is defined as 
\[T_{\psi}f=P(\psi f)\] 
for $f\in A^2(\D)$. 

The spectrum for $H^*_{\psi}H_{\psi}$, for symbols continuous up to the boundary,  
is a discrete set for a large class of domains on which the operator is compact 
\cite{CuckovicSahutoglu09}. So it would be interesting to know the sufficient 
conditions for the spectrum to contain intervals. 

Let $\sigma(T)$ denote the spectrum of a linear map $T$. The set of eigenvalues is called 
the point spectrum $\sigma_p(T)$ \cite[VII Definition 6.2]{ConwayBook}. The discrete spectrum, 
$\sigma_d(T)$ is composed of eigenvalues with finite (algebraic) multiplicity that are isolated 
points of $\sigma(T)$. Finally, the essential spectrum $\sigma_e(T)$ is defined as 
$\sigma_e(T)=\sigma(T)\setminus \sigma_d(T)$. A characterization of the essential 
spectrum in our set-up, called Weyl's Criterion, is presented in Theorem \ref{ThmWeyl} below. 
We note that for self-adjoint operators $T_1$ and $T_2$ we have $\sigma_e(T_1)=\sigma_e(T_2)$ 
if $T_1-T_2$ is compact (see, for instance, \cite[Theorem 14.6]{HislopBook} 
or \cite[XI Proposition 4.2]{ConwayBook}). 

Our first result is about the essential spectrum of $H^*_{\psi}H_{\psi}$ on $A^2(\mathbb{D}^n)$ 
when $\psi$ is a product of two functions that depend on different variables. We note that, 
$H^{\mathbb{D}^{n-1*}}_{\varphi}$ in the theorem below denotes the adjoint 
of the operator $H^{\mathbb{D}^{n-1}}_{\varphi}$.

\begin{theorem}\label{ThmProdSym} 
Let $\varphi\in C(\overline{\mathbb{D}^{n-1}}),\chi\in C(\overline{\mathbb{D}})$ and  
$\psi(z',z_n)=\varphi(z')\chi(z_n)$ for $z'\in \overline{\mathbb{D}^{n-1}},$ and 
$z_n\in \overline{\mathbb{D}}$. Then the essential spectrum of $H^*_{\psi}H_{\psi}$ 
contains  the set
\[\left\{|\chi(q)|^2\mu:q\in b\mathbb{D}, 
\mu\in \sigma(H^{\mathbb{D}^{n-1*}}_{\varphi}H^{\mathbb{D}^{n-1}}_{\varphi})\right\}.\]
\end{theorem}

In case the symbol is separable we have the following corollary,  proof of which will be presented 
in Section \ref{SecProof1}.

\begin{corollary}\label{CorProdSym} 
Let $\chi_j\in C(\overline{\mathbb{D}})$ for $j=1,\ldots, n$ and 
$\psi(z_1,\ldots,z_n)=\chi_1(z_1)\chi_2(z_2)\cdots\chi_n(z_n)$.  Then the essential spectrum of 
$H^*_{\psi}H_{\psi}$ contains the set  
\[\bigcup_{j=1}^n\left\{\mu\prod_{k\neq j}|\chi_k(q_k)|^2:q_k\in b\mathbb{D}, 
\mu\in \sigma(H^{\mathbb{D}^*}_{\chi_j}H^{\mathbb{D}}_{\chi_j})\right\}.\]
\end{corollary}

In the next theorem we give a sufficient condition for the spectrum to contain an interval 
for more general symbols than the ones in Theorem \ref{ThmProdSym}. 

\begin{theorem}\label{ThmGenSym}
Let $\psi \in C(\overline{\mathbb{D}^n}), 1\leq k\leq n-1,$ and $\psi_q(z')=\psi(z',q)$ 
for $z'\in \mathbb{D}^{n-1}$ and $q\in b\mathbb{D}$. Assume that 
$q\to \| H^{\mathbb{D}^{n-1}}_{\psi_q}\|$ is non-constant. Then the essential spectrum 
of $H^*_{\psi}H_{\psi}$ contains an open interval. 
\end{theorem}

The following corollary can be proved using Theorems \ref{ThmProdSym} 
or \ref{ThmGenSym}. We will mention both in Section \ref{SecProof1}.
\begin{corollary}\label{CorProd2Sym} 
Let $\varphi\in C(\overline{\mathbb{D}^{n-1}})$ and $\chi\in C(\overline{\mathbb{D}})$ 
such that $|\chi|$ is not constant on the unit circle and  $\varphi$ is not holomorphic. 
Then the essential spectrum of $H^*_{\psi}H_{\psi}$ contains an open interval where 
$\psi(z',z_n)=\varphi(z')\chi(z_n)$ for $z'\in \overline{\mathbb{D}^{n-1}}$ and 
$z_n\in \overline{\mathbb{D}}$. 
\end{corollary}

\begin{example}
Let $\psi(z_1,z_2)=\zb_1(\zb_2+1)$. Since $\zb_1$ is not holomorphic and $|\zb_2+1|$ is 
not constant on the unit circle, Corollary \ref{CorProd2Sym} implies that 
$\sigma_e(H^*_{\psi}H_{\psi})$ contains an interval. It is worth noting that  
(see Theorem \ref{ThmMonSym} below) the spectrum of the Hermitian square of a Hankel 
operator with a monomial symbol is a discrete set; yet for the a quadratic polynomial 
$\zb_1(\zb_2+1)$ the spectrum contains an interval.
\end{example}

In Theorem \ref{ThmMonSym} below, we give a complete characterization of the 
 spectrum on the polydisc for monomial symbols. 
 
We note that $\mathbb{N}_0$ denotes the set of non-negative integers, $\{0,1,2,\ldots\}$.
For $\mathbf{m},\mathbf{n},\Balpha\in \mathbb{N}^n_0$ and 
$\emptyset\neq B\subset B_n=\{1,2,3,\ldots, n\}$ we define 
\begin{align*}
\lambda_{\mathbf{n},\mathbf{m},\Balpha,B}=\begin{cases}
\prod_{k\in B}\frac{\alpha_k+1}{\alpha_k+n_k+m_k+1}, 
&  \alpha_k< m_k-n_k \text{ for some } k\in B \\
\prod_{k\in B}\frac{\alpha_k+1}{\alpha_k+n_k+m_k+1}
-\prod_{k\in B}\frac{(\alpha_k+1)(\alpha_k+n_k-m_k+1)}{(\alpha_k+n_k+1)^2}, 
&  \alpha_k\geq m_k-n_k \text{ for all } k\in B
\end{cases}.
\end{align*}

\begin{theorem}  \label{ThmMonSym} 
Let $\psi(z)=z^{\mathbf{n}}\zb^{\mathbf{m}}$ for some $\mathbf{m},\mathbf{n}\in \mathbb{N}^n_0$. 
Then $H^*_{\psi}H_{\psi}$ on $A^2(\mathbb{D}^n)$  has the following spectrum 
\[\sigma(H^*_{\psi}H_{\psi})=\{0\}\cup \{\lambda_{\mathbf{n},\mathbf{m},\Balpha,B}: 
\Balpha\in \mathbb{N}_0^n, \emptyset\neq B\subset B_n\}.\]
Furthermore, if $n_k+m_k\geq 1$ for all $k\in B_n$ and $m_k\geq 1$ for some $k\in B_n$ 
then all of the eigenvalues have finite multiplicities. On the other hand, if $n_k+m_k=0$ 
for some $k\in B_n$ then all of the eigenvalues have infinite multiplicities.
\end{theorem}

Since the spectrum of $H^*_{\psi}H_{\psi}$ on $A^2(\mathbb{D}^n)$ for 
$\psi(z)=z^{\mathbf{n}}\zb^{\mathbf{m}}$ contains countably many points 
it has empty interior. 
\begin{corollary}
Let $\psi(z)=z^{\mathbf{n}}\zb^{\mathbf{m}}$ for some $\mathbf{m},\mathbf{n}\in \mathbb{N}^n_0$. 
Then the spectrum of $H^*_{\psi}H_{\psi}$ on $A^2(\mathbb{D}^n)$ has empty interior.
\end{corollary}
In case the symbol is a pure conjugate holomorphic monomial in $\C^2$, we have the 
following corollary.
\begin{corollary}\label{CorMonSym}
Let $\psi(z_1,z_2)=\zb_1^n\zb_2^m$  on $\mathbb{D}^2$ for some positive integers  $n,m$.  Then 
\begin{align*}
\sigma(H^*_{\psi}H_{\psi})=\{0\}\cup&\bigcup_{\alpha_1=0}^{n-1}
\bigcup_{\alpha_2=0}^{m-1}\left\{\frac{\alpha_1+1}{\alpha_1+n+1},
\frac{\alpha_2+1}{\alpha_2+m+1}\right\} \\ 
&\bigcup_{\alpha_1=n}^{\infty}\bigcup_{\alpha_2=m}^{\infty}
\left\{\frac{n^2}{(\alpha_1+n+1)(\alpha_1+1)},\frac{m^2}{(\alpha_2+m+1)(\alpha_2+1)}\right\}\\
&\bigcup_{\alpha_1=n}^{\infty}\bigcup_{\alpha_2=m}^{\infty}
\left\{\frac{n^2(\alpha_2+1)^2+m^2(\alpha_1+1)^2-n^2m^2}{(\alpha_1+n+1)(\alpha_2+m+1)
(\alpha_1+1)(\alpha_2+1)}\right\}\\ 
&\bigcup_{\alpha_1=0}^{n-1}\bigcup_{\alpha_2=0}^{\infty}
\left\{\frac{(\alpha_1+1)(\alpha_2+1)}{(\alpha_1+n+1)(\alpha_2+m+1)}\right\} 
\bigcup_{\alpha_1=0}^{\infty}\bigcup_{\alpha_2=0}^{m-1}
\left\{\frac{(\alpha_1+1)(\alpha_2+1)}{(\alpha_1+n+1)(\alpha_2+m+1)}\right\}.
\end{align*}
Furthermore, all of the eigenvalues are of finite multiplicity. 
\end{corollary}

A precursor to the next corollary has appeared in \cite[Remark 4.2]{CuckovicSahutoglu18}.
\begin{corollary}\label{CorMonSym2}
Let $\psi(z_1,z_2)=\zb_1^n$ on $\mathbb{D}^2$ for some positive integer $n$.  Then 
\begin{align*}
\sigma(H^*_{\psi}H_{\psi})=\{0\}\cup\bigcup_{\alpha_1=0}^{n-1}
\left\{\frac{\alpha_1+1}{\alpha_1+n+1}\right\} \bigcup_{\alpha_1=n}^{\infty}
\left\{\frac{n^2}{(\alpha_1+n+1)(\alpha_1+1)}\right\}. 
\end{align*}
Furthermore, all of the eigenvalues are of infinite multiplicity.
\end{corollary}
\section{Proof of Theorems \ref{ThmProdSym} and  \ref{ThmGenSym}}\label{SecProof1}

In the proofs below $K^{\D}$ and $k^{\D}_z=K^{\D}_z/\|K^{\D}_z\|$ denote the Bergman 
kernel and the normalized Bergman kernel of $\D$, respectively. We will drop the 
superscript $\D$ when the domain is clear. 
 
To prove Theorems \ref{ThmProdSym} and \ref{ThmGenSym} we will need the following 
result. We refer the reader to Theorem 5.10 and Theorem 7.2 in \cite{HislopBook} for a proof.

\begin{theoremA}[Weyl's Criterion]\label{ThmWeyl} 
Let $A$ be a bounded self-adjoint linear operator on a Hilbert space $H$. Then 
\begin{itemize} 
\item[i.]  $\lambda$ is in the spectrum of $A$ if and only if 
there exists a sequence $\{u_n\}\subset H$ such that 
$\|u_n\|=1$ for all $n$ and $\|(A-\lambda )u_n\|\to 0$ as $n\to \infty$. 
\item [ii.] $\lambda$ is in the essential spectrum of $A$ if and only if 
there exists a sequence $\{u_n\}\subset H$ such that $\|u_n\|=1$ for all 
$n$, $u_n\to 0$ weakly, and $\|(A-\lambda )u_n\|\to 0$ as $n\to \infty$. 
\end{itemize} 
\end{theoremA}

The following lemma is probably well known. We include it here for completeness. 
We would like to thank Tomas Miguel P. Rodriguez for the proof. 
\begin{lemma}\label{LemaWeakConv} 
Let $\D$ be a domain in $\C^n$ and $\{f_j\}$ be a sequence in $A^2(\D)$. Then $f_j\to 0$ 
weakly as $j\to \infty$ if and only if $\{f_j\}$ is bounded in $A^2(\D)$ and $f_j\to 0$ as $j\to \infty$
uniformly on compact subsets in $\D$. 
\end{lemma} 
\begin{proof} 
Let us assume that $\{f_j\}$ is bounded in $A^2(\D)$ and $f_j\to 0$ as $j\to \infty$ uniformly 
on compact subsets in $\D$. Let $f\in A^2(\D)$ and $\ep>0$. Then there exists a compact set 
$K\subset \D$ such that $\|f\|_{L^2(\D\setminus K)}<\ep$. Then for large $j$  we have 
$\sup\{|f_j(z)|:z\in K\}<\ep$ and 
\begin{align*}
	|\langle f, f_j\rangle| 
	\leq \, & \|f\|_{L^2(K)}\| f_j\|_{L^2(K)}+\|f\|_{L^2(\D\setminus K)} \|f_j\|_{L^2(\D)} \\
	\leq \,&  \ep\left(\|f\|_{L^2(K)}\sqrt{V(K)}+\sup\{\|f_j\|_{L^2(\D)}:j\in \mathbb{N}\}\right) 
\end{align*}
where $V(K)$ denotes the Lebesgue volume of $K$. Hence $\langle f, f_j\rangle\to 0$. 
That is, $f_j\to 0$ weakly as $j\to \infty$. 

For the converse, we assume that  $f_j\to 0$ weakly as $j\to \infty$. We define 
$S_{f_j}(f)=\langle f,f_j\rangle$ for $f\in A^2(\D)$. Then  $S_{f_j}(f)\to 0$ as $j\to \infty$. That is 
$\sup \{|S_{f_j}(f)|:j\in \mathbb{N}\}<\infty$ for all $f\in A^2(\D)$. Then by the 
uniform boundedness principle 
$\sup \{\|S_{f_j}\|:j\in \mathbb{N}\}=\sup \{ \|f_j\|:j\in \mathbb{N}\}<\infty$. That is $\{f_j\}$ 
is bounded in $A^2(\D)$. Furthermore, $f_j(z)=\langle f_j, K_z\rangle\to 0$ as $j\to \infty$  
for all $z\in \D$. We will use this fact to conclude that $f_j\to 0$ uniformly on 
compact subsets as follows. Let $K$ be a compact subset of $\D$ and  
$\{f_{j_k}\}$ be a subsequence  of $\{f_j\}$. Then Montel's theorem implies that there is a 
further subsequence $\{f_{j_{k_l}}\}$ that is convergent to a holomorphic function $f$ uniformly on $K$. 
However, $f_j(z)=\langle f_j, K_z\rangle\to 0$ as $j\to \infty$. Then $f=0$. Therefore,  every subsequence 
$\{f_{j_k}\}$ has a further  subsequence  $\{f_{j_{k_l}}\}$ converging to zero uniformly on $K$. 
Then $f_j\to 0$ uniformly on $K$ as $j\to \infty$. Since $K$ is arbitrary we conclude that 
$f_j\to 0$ uniformly on compact subsets. 
\end{proof} 

\begin{proof}[Proof of Theorem \ref{ThmProdSym}] 
Let $ q\in b\mathbb{D}$. For $p\in\mathbb{D}$ and $j\in \mathbb{N}$ we define 
\[f_{j,p}(z',z_n)=g_j(z')k^{\mathbb{D}}_p(z_n)\] 
where $z'=(z_1,\ldots,z_{n-1}), k^{\mathbb{D}}_p$ is the normalized kernel for $A^2(\mathbb{D})$ 
centered at $p$ and $\{g_j\}\subset A^2(\mathbb{D}^{n-1})$, to be determined later, such that 
$\|g_j\|_{L^2(\mathbb{D}^{n-1})}=1$. We note that $\{g_j\}$ is uniformly bounded 
on compact subsets of $\mathbb{D}^{n-1}$. Then $\sup\{|f_{j,p}(z)|:j\in\mathbb{N},z\in K\} \to 0$ 
as  $p\to q$ for any compact set $K\subset \mathbb{D}^n$. Hence, by Lemma \ref{LemaWeakConv}, 
for any $p_j\to q$ we conclude that $f_{j,p_j}\to 0$ weakly $j\to\infty$.

Let us fix $j$ and denote $\psi_q(z',z_n)=\psi(z',q)=\chi(q)\varphi(z')$. 
Since $\psi_q$ is independent of $z_n$ and 
\[P(fg)(z)=(P^{\mathbb{D}^{n-1}}f)(z')\cdot (P^{\mathbb{D}}g)(z_n)\] 
whenever $f$ is a function of $z'$ and $g$ is a function of $z_n$, 
we have  
\begin{align} \nonumber 
H_{\psi_q}f_{j,p}
= \, & \psi_q g_jk^{\mathbb{D}}_p-P(\psi_qg_jk^{\mathbb{D}}_p) \\
\label{Eqn1}= \,& \chi(q)\varphi g_j k^{\mathbb{D}}_p
- \chi(q)\left(P^{\mathbb{D}^{n-1}}(\varphi g_j)\right)k^{\mathbb{D}}_p \\
\nonumber =\, &\chi(q)(H^{\mathbb{D}^{n-1}}_{\varphi}g_j)k^{\mathbb{D}}_p.
\end{align}
Then $\|H_{\psi_q}f_{j,p}\| 
=|\chi(q)|\|H^{\mathbb{D}^{n-1}}_{\varphi}g_j\|_{L^2(\mathbb{D}^{n-1})}.$ Now we 
write $\psi=\psi-\psi_q+\psi_q$ and 
\[H^*_{\psi}H_{\psi}f_{j,p}=H^*_{\psi_q}H_{\psi_q}f_{j,p} 
+H^*_{\psi}H_{\psi-\psi_q}f_{j,p} +H^*_{\psi-\psi_q}H_{\psi_q}f_{j,p}.\]
Hence for $\lambda\in\mathbb{R}$ we have 
\begin{align} \nonumber 
\left\|H^*_{\psi_q}H_{\psi_q}f_{j,p}-\lambda f_{j,p}\right\|
-&\left\|H^*_{\psi}H_{\psi-\psi_q}f_{j,p} +H^*_{\psi-\psi_q}H_{\psi_q}f_{j,p} \right\| \\
\label{Eqn2} \leq\,& \left\|H^*_{\psi}H_{\psi}f_{j,p}-\lambda f_{j,p}\right\| \\
 \nonumber \leq\, & \left\|H^*_{\psi_q}H_{\psi_q}f_{j,p}-\lambda f_{j,p}\right\|\\
\nonumber &+\left\|H^*_{\psi}H_{\psi-\psi_q}f_{j,p} +H^*_{\psi-\psi_q}H_{\psi_q}f_{j,p} \right\|. 
\end{align}

Let $\{h_j\}$ be a bounded sequence in $L^2(\mathbb{D}^{n-1})$. Then 
$\|(\psi-\psi_q)h_jk^{\mathbb{D}}_p\|\to 0$ as $p\to q$ for any fixed $j$. This  
can be seen as follows. By continuity of $\psi$, for $\ep>0$ there exists $\delta>0$ such that 
\[|\psi(z',z_n)-\psi_q(z',z_n)|<\ep  \text{ for } |z_n-q|<\delta.\] 
Then  we have 
\begin{align*}  
 &\left\|(\psi-\psi_q) h_jk^{\mathbb{D}}_p\right\|^2 \\
=\,& \left\|(\psi-\psi_q)h_jk^{\mathbb{D}}_p\right\|^2_{L^2(\{(z',z_n)\in \mathbb{D}^n: |z_n-q|< \delta\})}  
+\left\|(\psi-\psi_q)h_jk^{\mathbb{D}}_p\right\|^2_{L^2(\{(z',z_n)\in \mathbb{D}^n: |z_n-q|\geq \delta\})}\\
 \leq\, & \ep^2 \|h_j\|^2_{L^2(\mathbb{D}^{n-1})}\|k^{\mathbb{D}}_p\|^2_{L^2(\mathbb{D})} \\
&+ \pi \|h_j\|^2_{L^2(\mathbb{D}^{n-1})}  \sup\left\{\left|(\psi(z',z_n)-\psi_q(z',z_n))k^{\mathbb{D}}_p(z_n)\right|^2: 
	z'\in \mathbb{D}^{n-1}, |z_n-q|\geq \delta\right\}. 
\end{align*}
However, 
\[\sup\left\{\left|k^{\mathbb{D}}_p(z_n)\right|:|z_n-q|\geq \delta\right\} \to 0 \text{ as } p\to q.\] 
Then, for any $j$ we have $\limsup_{p\to q}\|(\psi-\psi_q) h_jk^{\mathbb{D}}_p\|\leq \ep\|h_j\|$ for all $\ep>0$. 
Therefore,  
\begin{align} \label{Eqn3} 
\lim_{p\to q}\left\|(\psi-\psi_q) h_jk^{\mathbb{D}}_p\right\| =0 \text{ for any } j.  
\end{align}  
This fact together with $h_j=1$ imply that for any $j$ we have  
\[\left\|H^*_{\psi}H_{\psi-\psi_q}f_{j,p}\right\| = 
\left\|H^*_{\psi}(I-P)\Big((\psi-\psi_q)f_{j,p}\Big)\right\|\to 0 \text{ as }p \to q.\] 

Next we will use the following fact (see \cite[Lemma 1]{CuckovicSahutoglu14}) 
\begin{align}\label{EqnHankelProd}
H^*_u H_v =  PM_{\overline{u}}(I-P)M_v \text{ for } u,v\in L^{\infty}(\D)
\end{align} 
where $M_v$ denotes the multiplication operator by $v$.

Since $\mathbb{D}^n$ is a product domain and $\psi_q$ is independent of $z_n$, 
the fact above, \eqref{Eqn1} and \eqref{Eqn3} with $h_j=H^{\mathbb{D}^{n-1}}_{\varphi}g_j$ imply that 
for any $j$ we have 
\[\left\|H^*_{\psi-\psi_q}H_{\psi_q}f_{j,p}\right\|
=|\chi(q)|\left\|P\left((H^{\mathbb{D}^{n-1}}_{\varphi}g_j)(\overline{\psi} 
-\overline{\psi_q})k^{\mathbb{D}}_p\right)\right\|\to 0 \text{ as } p\to q\]
 and 
\[H^*_{\psi_q}H_{\psi_q}f_{j,p}
= |\chi(q)|^2(H^{\mathbb{D}^{n-1*}}_{\varphi}H^{\mathbb{D}^{n-1}}_{\varphi}g_j)k^{\mathbb{D}}_p.\]
Then using \eqref{Eqn2} for any $j$ we have   
\begin{align}\label{EqnLim}  
\left\|H^*_{\psi}H_{\psi}f_{j,p}-\lambda f_{j,p}\right\| 
\to  \left\||\chi(q)|^2H^{\mathbb{D}^{n-1*}}_{\varphi}H^{\mathbb{D}^{n-1}}_{\varphi}g_j 
-\lambda g_j\right\|_{L^2(\mathbb{D}^{n-1})} \text{ as } p\to q. 
\end{align}

Let $\mu\in \sigma(H^{\mathbb{D}^{n-1*}}_{\varphi}H^{\mathbb{D}^{n-1}}_{\varphi})$. Then, 
by i. in Theorem \ref{ThmWeyl}, we choose a sequence $\{g_j\}\subset A^2(\mathbb{D}^{n-1})$ 
such that $\|g_j\|=1$ for all $j$ and 
\[\left\|H^{\mathbb{D}^{n-1*}}_{\varphi}H^{\mathbb{D}^{n-1}}_{\varphi}g_j 
-\mu g_j\right\|_{L^2(\mathbb{D}^{n-1})}\to 0 \text{ as } j\to \infty. \] 
Then for $\lambda=|\chi(q)|^2\mu$ we have 
\[\left\||\chi(q)|^2H^{\mathbb{D}^{n-1*}}_{\varphi}H^{\mathbb{D}^{n-1}}_{\varphi}g_j 
-\lambda g_j\right\|_{L^2(\mathbb{D}^{n-1})}\to 0 \text { as } j\to\infty.\] 
Then  using \eqref{EqnLim}  and the limit above we choose $p_j\in \mathbb{D}$ 
such that $p_j\to q$ as $j\to \infty$ and  
\[\left\|H^*_{\psi}H_{\psi}f_{j,p_j}-\lambda f_{j,p_j}\right\|
< \left\||\chi(q)|^2H^{\mathbb{D}^{n-1*}}_{\varphi}H^{\mathbb{D}^{n-1}}_{\varphi}g_j 
-\lambda g_j\right\|_{L^2(\mathbb{D}^{n-1})}+\frac{1}{j}\] 
 for all  $j$. Hence 
\[\left\|H^*_{\psi}H_{\psi}f_{j,p_j}-\lambda f_{j,p_j}\right\| \to 0 \text{ as } j\to \infty.\] 
Finally, we use ii. in Theorem \ref{ThmWeyl} to conclude that 
$\lambda=|\chi(q)|^2\mu\in \sigma_e(H^*_{\psi}H_{\psi})$ for any $q\in b\mathbb{D}$ 
and $\mu\in \sigma(H^{\mathbb{D}^{n-1*}}_{\varphi}H^{\mathbb{D}^{n-1}}_{\varphi})$.
\end{proof}

\begin{proof}[Proof of Corollary \ref{CorProdSym}]
Without loss of generality let $\mu\in \sigma(H^{\mathbb{D}^*}_{\chi_1}H^{\mathbb{D}}_{\chi_1})$. 
Then,  by Theorem \ref{ThmProdSym} for $n=2$,  we have  
$\mu |\chi_2(q_2)|^2 \in \sigma(H^{\mathbb{D}^{2*}}_{\chi_1\chi_2}H^{\mathbb{D}^2}_{\chi_1\chi_2})$. 
Hence, inductively, we get  
\[\mu |\chi_2(q_2)\cdots\chi_n(q_n)|^2\in \sigma_e(H^*_{\psi}H_{\psi})\]
completing the proof of the corollary. 
\end{proof}

The proof of Theorem \ref{ThmGenSym} is similar to the proof of Theorem \ref{ThmProdSym}. 
So we will not provide all the details but will highlight the differences and similarities below.

\begin{proof}[Proof of Theorem \ref{ThmGenSym}] 
Let $q\in b\mathbb{D}$ and  
\[\lambda_q
=\left\|H^{\mathbb{D}^{n-1*}}_{\psi_q}H^{\mathbb{D}^{n-1}}_{\psi_q}\right\|_{L^2(\mathbb{D}^{n-1})}
=\left\|H^{\mathbb{D}^{n-1}}_{\psi_q}\right\|^2_{L^2(\mathbb{D}^{n-1})}.\]  
Then 
$\lambda_q\in \sigma(H^{\mathbb{D}^{n-1*}}_{\psi_q}H^{\mathbb{D}^{n-1}}_{\psi_q})$ 
(see, for instance, \cite[Theorem 5.14]{HislopBook}). By i. in Theorem \ref{ThmWeyl} there 
exists $\{g_{j,q}\}\subset A^2(\mathbb{D}^{n-1})$ such that $\|g_{j,q}\|=1$ and 
\begin{align}\label{Eqn6} 
\left\|H^{\mathbb{D}^{n-1*}}_{\psi_q}H^{\mathbb{D}^{n-1}}_{\psi_q} g_{j,q} 
-\lambda_qg_{j,q}\right\|_{L^2(\mathbb{D}^{n-1})}\to 0 \text{ as } j\to\infty.
\end{align} 
 
Let  
\[f_{j,q,p}(z',z_n)=g_{j,q}(z')k^{\mathbb{D}}_p(z_n)\] 
(again here $k^{\mathbb{D}}_p$ is the normalized kernel for $\mathbb{D}$ centered at $p$). 
As in the proof of Theorem \ref{ThmProdSym} we have $f_{j,q,p_j}\to 0$ weakly as $j\to \infty$ 
for any $p_j\to q$. Similar to \eqref{Eqn1}, one can show that  
\[H_{\psi_q}f_{j,q,p}= (H^{\mathbb{D}^{n-1}}_{\psi_q}g_{j,q})k^{\mathbb{D}}_p.\]
Then $\|H_{\psi_q}f_{j,q,p}\| 
=\|H^{\mathbb{D}^{n-1}}_{\psi_q}g_{j,q}\|_{L^2(\mathbb{D}^{n-1})}$ and as in \eqref{Eqn2} we have 
\begin{align*} 
\left\|H^*_{\psi_q}H_{\psi_q}f_{j,q,p}-\lambda_q f_{j,q,p}\right\|
-&\left\|H^*_{\psi}H_{\psi-\psi_q}f_{j,q,p} +H^*_{\psi-\psi_q}H_{\psi_q}f_{j,q,p} \right\| \\
\leq\,& \left\|H^*_{\psi}H_{\psi}f_{j,q,p}-\lambda_qf_{j,q,p}\right\| \\
\leq\, & \left\|H^*_{\psi_q}H_{\psi_q}f_{j,q,p} -\lambda_q f_{j,q,p}\right\| \\ 
&+\left\|H^*_{\psi}H_{\psi-\psi_q}f_{j,q,p} +H^*_{\psi-\psi_q}H_{\psi_q}f_{j,q,p} \right\|. 
\end{align*}
As in the proof of Theorem \ref{ThmProdSym} one can show that $\|(\psi-\psi_q)f_{j,q,p}\|\to 0$
 as $p\to q$. This fact implies that    
\[\left\|H^*_{\psi}H_{\psi-\psi_q}f_{j,q,p}\right\| 
=\left\|H^*_{\psi}(I-P)\Big((\psi-\psi_q)f_{j,q,p}\Big)\right\|\to 0 \text{ as } p\to q.\] 
Furthermore, since $\mathbb{D}^n$ is a product domain and $\psi_q$ is independent of 
$z_n$ using \eqref{Eqn3} with 
\[h_{j,q}=H^{\mathbb{D}^{n-1}}_{\psi_q}g_{j,q}\] 
we have 
\[\left\|H^*_{\psi-\psi_q}H_{\psi_q}f_{j,q,p}\right\| 
=\left\|P\left((H^{\mathbb{D}^{n-1}}_{\psi_q}g_{j,q})(\overline{\psi}-\overline{\psi_q}) 
k^{\mathbb{D}}_{p}\right)\right\| \to 0 \text{ as } p\to q.\] 
Furthermore, by \eqref{Eqn1} we have 
\[H^*_{\psi_q}H_{\psi_q}f_{j,q,p}  
= (H^{\mathbb{D}^{n-1*}}_{\psi_q}H^{\mathbb{D}^{n-1}}_{\psi_q}g_{j,q})k^{\mathbb{D}}_p.\]
Then for any fixed $j$ we have 
\begin{align*} 
\left\|H^*_{\psi}H_{\psi}f_{j,q,p} -\lambda_q f_{j,q,p}\right\| 
\to  \left\|H^{\mathbb{D}^{n-1*}}_{\psi_q}H^{\mathbb{D}^{n-1}}_{\psi_q}g_{j,q} 
-\lambda_q g_{j,q}\right\|_{L^2(\mathbb{D}^{n-1})} \text { as } p\to q.  
\end{align*}
However, by \eqref{Eqn6} we can pass to subsequence of $\{g_{j,q}\}$, if necessary, and get  
\[\left\|H^{\mathbb{D}^{n-1*}}_{\psi_q}H^{\mathbb{D}^{n-1}}_{\psi_q} g_{j,q} 
-\lambda_qg_{j,q}\right\|_{L^2(\mathbb{D}^{n-1})}<\frac{1}{j}\] 
for all $j$. Next we choose  $\{p_j\}\subset \mathbb{D}$ such that $p_j\to q$ 
as $j\to \infty$ and 
\[\left\|H^*_{\psi}H_{\psi}f_{j,q,p_j} -\lambda_q f_{j,q,p_j}\right\|<\frac{2}{j}\] 
for all $j$. Therefore, we have 
\[\left\|H^*_{\psi}H_{\psi}f_{j,q,p_j} -\lambda_q f_{j,q,p_j}\right\| \to 0 \text { as } j\to \infty.\]
Then ii. in Theorem \ref{ThmWeyl} implies that $\lambda_q\in \sigma_e(H^*_{\psi}H_{\psi})$. 

Next we note that $\|H^{\mathbb{D}^{n-1*}}_{\psi_q}H^{\mathbb{D}^{n-1}}_{\psi_q}\|$ depends 
on $q$ continuously because $\psi_q$ depends on $q$ continuously and 
\begin{align*}
\left\|H^{\mathbb{D}^{n-1*}}_{\psi_{q_1}}H^{\mathbb{D}^{n-1}}_{\psi_{q_1}}
-H^{\mathbb{D}^{n-1*}}_{\psi_{q_2}}H^{\mathbb{D}^{n-1}}_{\psi_{q_2}}\right\| 
= & \left\|H^{\mathbb{D}^{n-1*}}_{\psi_{q_1}-\psi_{q_2}}H^{\mathbb{D}^{n-1}}_{\psi_{q_1}} 
+H^{\mathbb{D}^{n-1*}}_{\psi_{q_2}}H^{\mathbb{D}^{n-1}}_{\psi_{q_1}-\psi_{q_2}} \right\| \\
\leq \,& \left\|H^{\mathbb{D}^{n-1}}_{\psi_{q_1}-\psi_{q_2}}\right\|
\left(\left\|H^{\mathbb{D}^{n-1}}_{\psi_{q_1}}\right\| 
+\left\|H^{\mathbb{D}^{n-1}}_{\psi_{q_2}}\right\|\right)\\
\leq \, &\left\| \psi_{q_1}-\psi_{q_2}\right\|_{L^{\infty}} 
\left(\left\|H^{\mathbb{D}^{n-1}}_{\psi_{q_1}}\right\| 
+\left\|H^{\mathbb{D}^{n-1}}_{\psi_{q_2}}\right\|\right). 
\end{align*} 
Then we conclude that $\{\lambda_q:q\in b\mathbb{D}\}$ is connected. Therefore, 
since we assume that the mapping  $q\to \| H^{\mathbb{D}^{n-1}}_{\psi_q}\|$ is non-constant, 
we conclude that $\{\lambda_q:q\in b\mathbb{D}\}$ contains an open interval in $(0, \infty)$. 
That is, $\sigma_e(H^*_{\psi}H_{\psi})$ contains an open interval in $(0, \infty)$.
 \end{proof}

The proof of Theorem \ref{ThmGenSym} above implies the following corollary.

\begin{corollary}\label{Cor1} 
Let $\psi \in C(\overline{\mathbb{D}^2})$ and  
$\psi_{1,\theta}(\xi)=\psi(e^{i\theta},\xi), \psi_{2,\theta}(\xi)=\psi(\xi,e^{i\theta})$ for 
$\xi\in \mathbb{D}$. Then 
\[\left\{\left\| H^{\mathbb{D}}_{\psi_{1,\theta}}\right\|^2_{L^2(\mathbb{D})}:\theta\in [0,2\pi]\right\}
\cup \left\{\left\| H^{\mathbb{D}}_{\psi_{2,\theta}}\right\|^2_{L^2(\mathbb{D})}:\theta\in [0,2\pi]\right\} 
\subset  \sigma_e(H^*_{\psi}H_{\psi}).\] 
\end{corollary}

\begin{remark}
We note that Theorem \ref{ThmMonSym} shows that the inclusion in Corollary \ref{Cor1} 
is not an equality in general. Indeed, for $\psi(z_1,z_2)=\zb_1^n\zb_2^m$ the quantities 
$\| H^{\mathbb{D}}_{\psi_{1,\theta}}\|$ and $\| H^{\mathbb{D}}_{\psi_{2,\theta}}\|$ are constant as 
functions of $\theta$. Therefore, the left hand side of the inclusion in Corollary \ref{Cor1} 
contains at most two numbers whereas the right hand side, by Theorem \ref{ThmMonSym}, 
contains infinitely many numbers.
\end{remark}

\begin{proof}[Proof of Corollary \ref{CorProd2Sym}]
If $\varphi$ is not holomorphic then $H^{\mathbb{D}^{n-1}}_{\varphi}$ is non-zero operator 
as $H^{\mathbb{D}^{n-1}}_{\varphi}1\neq 0$. Hence $\|H^{\mathbb{D}^{n-1}}_{\varphi}\|>0$. 
Furthermore,  Since $|\chi|$ is non-constant on the unit circle, the image of the mapping 
\[q\to \left\| H^{\mathbb{D}^{n-1}}_{\psi_q}\right\|_{L^2(\mathbb{D}^{n-1})} 
=|\chi(q)|\left\|H^{\mathbb{D}^{n-1}}_{\varphi}\right\|_{L^2(\mathbb{D}^{n-1})}\] 
contains an open interval. Therefore, by Theorem \ref{ThmGenSym}, the essential spectrum of 
$H^*_{\psi}H_{\psi}$ contains an open interval. 
\end{proof}

\begin{remark}
One can also prove of Corollary \ref{CorProd2Sym} using Theorem \ref{ThmProdSym} as 
follows. Since $H^{\mathbb{D}^{n-1}}_{\varphi}$ is non-zero operator,  there exists a positive 
number $\mu\in \sigma(H^{\mathbb{D}^{n-1*}}_{\varphi}H^{\mathbb{D}^{n-1}}_{\varphi})$. 
Then the image of the mapping $q\to |\chi(q)|^2\mu$ contains an open interval as 
$|\chi|$ is non-constant on the unit circle. Therefore,  by Theorem \ref{ThmProdSym}, 
the essential spectrum of $H^*_{\psi}H_{\psi}$ contains an open interval. 
\end{remark}

\section{Proof of Theorem \ref{ThmMonSym}}
Before we present the proof of Theorem \ref{ThmMonSym} we make some elementary 
computations. Let $\Omega$ be a complete Reinhardt domain in $\mathbb C^n$ and 
again $\mathbb{N}_0=\{0,1,2,3,\ldots\}$. The Bergman kernel function $K$ has the expression
\[K(z,w)=\sum_{\Balpha\in \mathbb{N}_0^n}c_{\Balpha} z^{\Balpha}\wb^{\Balpha}\] 
where $c_{\Balpha}= 0$ when $ \|z^{\Balpha}\|_{L^2(\D)}=\infty$ and  
$c_{\Balpha}= \|z^{\Balpha}\|_{L^2(\D)}^{-2}$ otherwise. 

A function $\psi$ is called quasi-homogeneous if there exists 
$f:[0,\infty)^n\to \C$, $(k_1,\ldots,k_n)\in \mathbb{Z}^n$ such that 
\begin{equation}\label{Eqn4}
\psi(r_1e^{i\theta_1},\ldots,r_ne^{i\theta_n})
=f(r_1, \ldots,r_n)e^{i(k_1\theta_1+\cdots +k_n\theta_n)}.
\end{equation}
For $z=(z_1,\dots,z_n)\in \mathbb C^n$ and $\Balpha=(\alpha_1,\dots,\alpha_n)\in \mathbb Z^n$, 
we will use the notation $|z|$ for $(|z_1|,\dots,|z_n|)$ and write $|z|^{\Balpha}$ for the product 
$\prod_{j=1}^{n}|z_j|^{\alpha_j}$. Then \eqref{Eqn4} can also be expressed as:
\[\psi(z) =f(|z|)e^{i\mathbf{k}\cdot\Btheta},\]
where $\mathbf{k}=(k_1,\dots,k_n)$, $\Btheta=(\theta_1,\dots,\theta_n)$, and $z_j=|z_j|e^{i\theta_j}$.
Now we consider the spectrum of $H^*_{\psi}H_{\psi}$ when $\psi$ is bounded and 
quasi-homogeneous on the complete Reinhardt domain $\Omega$.

By \eqref{EqnHankelProd} we have 
\begin{align*}
H^*_{\psi} H_{\psi} z^{\Balpha}
&=PM_{\overline{\psi}}(I-P)M_\psi(z^{\Balpha})\\
&=PM_{\overline{\psi}}(I-P)f(|z|)|z|^{\Balpha} e^{i(\mathbf{k}+\Balpha)\cdot \Btheta}\\ 
&=PM_{\overline{\psi}}(z^{\Balpha}\psi-P(|z|^{\Balpha} e^{i(\mathbf{k}+\Balpha )\cdot \Btheta}f(|z|))) 
\end{align*}
for $\Balpha\in \mathbb N_0^n$. If the multi-index $\mathbf{k}+\Balpha\notin \mathbb{N}_0^n$, 
then $P(|z|^{\Balpha} e^{i(\mathbf{k}+\Balpha )\cdot \Btheta}f(|z|))=0$ which yields that 
\begin{align}\label{Eqn5} 
H^*_{\psi} H_{\psi} z^{\Balpha}
=P(z^{\Balpha}|\psi|^2)
=\frac{\|z^{\Balpha}\psi\|^2_{L^2}}{\|z^{\Balpha}\|^2_{L^2}}z^{\Balpha}.
\end{align} 
Otherwise, 
\begin{align*}
P(|z|^{\Balpha} e^{i(\mathbf{k}+\Balpha )\cdot \Btheta}f(|z|))
&=\int_{\Omega}K(z,w)|w|^{\Balpha} e^{i(\mathbf{k}+\Balpha )\cdot \Btheta}f(|w|)dV(w)\\
&=\sum_{\Bbeta\in \mathbb{N}_0^n}c_{\Bbeta}z^{\Bbeta} \int_{\Omega} \wb^{\Bbeta}  
	|w|^{\Balpha} e^{i(\mathbf{k}+\Balpha )\cdot \Btheta}f(|w|)dV(w)\\
&=c_{\Balpha+\mathbf{k}} z^{\Balpha+\mathbf{k}}\int_{\Omega}\wb^{\Balpha+\mathbf{k}}
	|w|^{\Balpha} e^{i(\mathbf{k}+\Balpha )\cdot \Btheta}f(|w|)dV(w)\\
&=\frac{\int_{\D}|w|^{2\Balpha+k} f(|w|)dV(w)}{\|z^{\Balpha+\mathbf{k}}\|^2_{L^2}}
z^{\mathbf{k}+\Balpha}, 
\end{align*}
which implies that
\begin{align}\label{Eqn8}
\nonumber H^*_{\psi} H_{\psi} z^{\Balpha} 
=\,&PM_{\overline{\psi}}(z^{\Balpha}\psi-P(|z|^{\Balpha} e^{i(\mathbf{k}+\Balpha )\cdot \Btheta}f(|z|)))\\
=\,&PM_{\overline{\psi}}\left(z^{\Balpha}\psi-\frac{\int_{\D}|w|^{2\Balpha+k} f(|w|)dV(w)}{\|z^{\Balpha 
+\mathbf{k}}\|^2_{L^2}}z^{\mathbf{k}+\Balpha}\right)\nonumber\\
=\,&P\left(z^{\Balpha} |\psi|^2-\frac{\int_{\D}|w|^{2\Balpha+k} f(|w|)dV(w)}{\|z^{\Balpha 
+\mathbf{k}}\|^2_{L^2}}|z|^{\mathbf{k}}z^{\Balpha}\overline{f(|z|)}\right)\nonumber\\
=\,&z^{\Balpha}\left(\frac{\|z^{\Balpha} \psi\|^2_{L^2}}{\|z^{\Balpha}\|^2_{L^2}} 
	-\frac{\left|\int_{\D}|w|^{2\Balpha+k}  f(|w|)dV(w)\right|^2}{\|z^{\Balpha}\|^2_{L^2}\| 
z^{\Balpha+\mathbf{k}}\|^2_{L^2}}\right).
\end{align}
Hence the spectrum $\sigma(H^*_{\psi}H_{\psi})$ contains the eigenvalues 
\[\frac{\|z^{\Balpha} \psi\|^2_{L^2}}{\|z^{\Balpha}\|^2_{L^2}} 
	-\frac{\left|\int_{\D}|w|^{2\Balpha+k} f(|w|)dV(w)\right|^2}{\|z^{\Balpha}\|^2_{L^2}\| 
z^{\Balpha+\mathbf{k}}\|^2_{L^2}}\]
corresponding to the eigenfunction $z^{\Balpha}$ for $\mathbf{k}+\Balpha\in \mathbb{N}_0^n$ and 
\[\frac{\|z^{\Balpha} \psi\|^2_{L^2}}{\|z^{\Balpha}\|^2_{L^2}}\]
for $\mathbf{k}+\Balpha\not\in \mathbb{N}_0^n$. 
Furthermore, since the eigenfunctions form an orthogonal basis for $A^2(\D)$, the closure of 
this set is the whole spectrum $\sigma(H^*_{\psi}H_{\psi})$ (see by Lemma \ref{LemClosure} below). 

\begin{example}
When $f(|z|)=|z|^{\mathbf{k}}$ for $\mathbf{k}\in \mathbb{N}_0^n$, equation \eqref{Eqn8} becomes
\[H^*_{\psi} H_{\psi} z^{\Balpha} 
=z^{\Balpha}\left(\frac{\|z^{\Balpha+\mathbf{k}} \|^2_{L^2}}{\|z^{\Balpha}\|^2_{L^2}}-
\frac{\|z^{2\Balpha+2\mathbf{k}}\|^2_{L^1}}{\|z^{\Balpha}\|^2_{L^2} 
 \|z^{\Balpha+\mathbf{k}}\|^2_{L^2}}\right)=0.\]
This is not surprising since in this case the symbol $\psi(z)=z^{\mathbf{k}}$ is holomorphic. 
\end{example}

\begin{example}
When $\psi(z)=e^{i\mathbf{k}\cdot \Btheta}$ for $\mathbf{k}\in \mathbb{N}_0^n$, equation 
\eqref{Eqn8} becomes
\[H^*_{\psi} H_{\psi} z^{\Balpha} 
=z^{\Balpha}\left(1 -\frac{\|z^{2\Balpha+\mathbf{k}}\|^2_{L^1}}{\|z^{\Balpha}\|^2_{L^2} 
 \|z^{\Balpha+\mathbf{k}}\|^2_{L^2}}\right).\]
\end{example}

\begin{example}
When $\mathbf{k}=\mathbf{0}$, $\psi(z)$ is radial and \eqref{Eqn8} becomes
\[H^*_{\psi} H_{\psi} z^{\Balpha} 
=z^{\Balpha} \left(\frac{\|z^{\Balpha}\psi \|^2_{L^2}}{\|z^{\Balpha}\|^2_{L^2}} 
-\frac{\|z^{2\Balpha}\psi\|^2_{L^1}}{\|z^{\Balpha}\|^4_{L^2}}\right).\]
\end{example}

\begin{lemma}\label{LemClosure} 
Let $T:H\to H$ be a bounded linear map on a separable Hilbert space $H$. Assume 
that $\lambda_j$ is an eigenvalue with a corresponding eigenvector $u_j$ for  
$j\in \mathbb{N}$ and $\{u_j:j\in \mathbb{N}\}$ forms an  orthogonal basis 
for $H$. Then $\sigma(T)=\overline{\{\lambda_j:j\in \mathbb{N}\}}$. 
Furthermore, $\lambda\in \sigma_e(T)$ if and only if it is the limit 
of a subsequence of $\{\lambda_j\}$.
\end{lemma}

\begin{proof} 
Since $\sigma(T)$ is a compact set containing $\{\lambda_j:j\in \mathbb{N}\}$ we 
have $\overline{\{\lambda_j:j\in \mathbb{N}\}}\subset \sigma(T)$. To prove the 
converse, assume that $\lambda\not\in \overline{\{\lambda_j:j\in \mathbb{N}\}}$. 
Then we define $Su_j=\frac{1}{\lambda_j-\lambda}u_j$ for all $j$. Then $S$ is a 
bounded operator, as $\{(\lambda_j-\lambda)^{-1}\}$ is a bounded sequence, and 
it is the inverse of $T-\lambda I$, as \[(T-\lambda I)Su_j=S(T-\lambda I)u_j=u_j\]
for all $j$. Then, $\lambda \not\in \sigma(T)$. That is, 
$\sigma(T)\subset  \overline{\{\lambda_j:j\in \mathbb{N}\}}$. 

We recall that $\sigma_e(T)=\sigma(T)\setminus \sigma_d(T)$ where 
 $\sigma_d(T)$ is composed of isolated eigenvalues with finite multiplicity. 
To prove the last statement, let $\lambda\in \sigma_e(T)$. Then $\lambda\in \sigma(T)$ 
is not an isolated eigenvalue with finite multiplicity. That is, either $\lambda$ 
is not isolated in $\sigma(T)$ or it is an eigenvalue with infinite multiplicity. 
Either way,  there exists a subsequence of $\{\lambda_j\}$ converging to $\lambda$. 
Conversely, if $\lambda\in \sigma(T)$ is the limit of a subsequence $\{\lambda_{j_k}\}$ of 
$\{\lambda_j\}$, then either $\lambda=\lambda_{j_k}$ for infinitely many $k$s 
(hence, it is an eigenvalue with infinite multiplicity) or $\lambda$ 
is not isolated in $\sigma(T)$. Again, either way $\lambda \in \sigma_e(T)$. 
\end{proof}

Finally, we present the proof of Theorem \ref{ThmMonSym}.   

\begin{proof}[Proof of Theorem \ref{ThmMonSym}]
Note that 
$\psi(z) =z^{\mathbf{n}}\zb^{\mathbf{m}} 
=|z|^{\mathbf{n}+\mathbf{m}}e^{i(\mathbf{n}-\mathbf{m})\cdot \Btheta}$. 
By \eqref{Eqn5} and \eqref{Eqn8}, it follows that 
\begin{align}\label{Eqn9}
H^*_{\psi} H_{\psi} z^{\Balpha}
=\begin{cases}
z^{\Balpha} \frac{\|z^{\Balpha+\mathbf{n}}\zb^{\mathbf{m}}\|^2_{L^2}}{\|z^{\Balpha}\|^2_{L^2}}
	& \Balpha+\mathbf{n}-\mathbf{m}\notin \mathbb{N}_0^n\\
z^{\Balpha}\left(\frac{\|z^{\Balpha+\mathbf{n}}\zb^{\mathbf{m}}\|^2_{L^2}}{\|z^{\Balpha}\|^2_{L^2}} 
-\frac{\|z^{2\Balpha+2\mathbf{n}-\mathbf{m}}\zb^{\mathbf{m}}\|^2_{L^1}}{\|z^{\Balpha}\|^2_{L^2} 
 \|z^{\Balpha+\mathbf{n}-\mathbf{m}}\|^2_{L^2}}\right)
	& \Balpha+\mathbf{n}-\mathbf{m}\in \mathbb N_0^n.
\end{cases}
\end{align}
We note that for multi-index $\Bbeta\in \mathbb N_0^n$ we have, 
\[\|z^{\Bbeta}\|^2_{L^2}=\|z^{2\Bbeta}\|_{L^1}=\frac{\pi^n}{\prod_{k=1}^{n}(\beta_k+1)}.\]
Substituting this into \eqref{Eqn9} yields that
\begin{align*}
H^*_{\psi} H_{\psi} z^{\Balpha} 
=\begin{cases}
z^{\Balpha}\prod_{k=1}^{n}\frac{\alpha_k+1}{\alpha_k+n_k+m_k+1}
& \Balpha+\mathbf{n}-\mathbf{m}\notin \mathbb{N}_0^n\\
z^{\Balpha}\left(\prod_{k=1}^{n}\frac{\alpha_k+1}{\alpha_k+n_k+m_k+1}
- \prod_{k=1}^{n}\frac{(\alpha_k+1)(\alpha_k+n_k-m_k+1)}{(\alpha_k+n_k+1)^2}\right) 
& \Balpha+\mathbf{n}-\mathbf{m}\in \mathbb{N}_0^n. 
\end{cases}
\end{align*}
Hence 
\[H^*_{\psi} H_{\psi} z^{\Balpha}=\lambda_{\mathbf{n},\mathbf{m},\Balpha,B_n}z^{\Balpha}\]  
for $\Balpha\in \mathbb{N}_0^n$ where  $\lambda_{\mathbf{n},\mathbf{m},\Balpha,B_n}$ 
is defined before Theorem \ref{ThmMonSym} and  $B_n=\{1,2,3,\ldots, n\}.$ 
That is, $z^{\Balpha}$ is an eigenvector corresponding to the eigenvalue  
$\lambda_{\mathbf{n},\mathbf{m},\Balpha,B_n}$. 
Furthermore, since $\{z^{\Balpha}:\Balpha\in \mathbb{N}_0^n\}$ forms an orthogonal 
basis for $A^2(\mathbb{D}^n)$, Lemma \ref{LemClosure} implies that 
\[\sigma(H^*_{\psi} H_{\psi}) =\overline{\left\{\lambda_{\mathbf{n},\mathbf{m},\Balpha,B_n}: 
\Balpha\in \mathbb{N}_0^n\right\}}.\]
We note that in case $\mathbf{m}=0$ the operator $H^*_{\psi} H_{\psi} z^{\Balpha}$ 
is the zero operator and hence 0 is the only eigenvalue with infinite multiplicity. 
For the rest of the proof we will  assume that $\mathbf{m}\neq 0$. 

Lemma \ref{LemClosure} implies that each element in the spectrum is either 
an eigenvalue or a limit of a sequence of eigenvalues. To describe the spectrum 
outside of the point spectrum  we assume that 
$\lambda \in \sigma(H^*_{\psi} H_{\psi})$ such that 
\[\lambda=\lim_{j\to\infty}\lambda_{\mathbf{n},\mathbf{m},\Balpha(j),B_n} \] 
for a sequence $\{\Balpha(j)\}\subset \mathbb{N}_0^n$. Next we will show that 
either $\lambda=0$ or $\lambda=\lambda_{\mathbf{n},\mathbf{m},\Balpha,B}$ 
for some  $\emptyset\neq B\subset B_n$ and $\Balpha\in \mathbb{N}_0^n$. 
We assume that $\lambda\neq 0$.

We note that $z^{\Balpha(j)}$ is an  eigenvector corresponding to the 
eigenvalue  $\lambda_{\mathbf{n},\mathbf{m},\Balpha(j),B_n}$. Let us assume that 
\[\lambda_{\mathbf{n},\mathbf{m},\Balpha(j),B_n} 
=\prod_{k=1}^{n} \frac{\alpha_k(j)+1}{\alpha_k(j)+n_k+m_k+1}.\]
For each $j$ there exists $k\in B_n$ such that $\alpha_k(j)<m_k-n_k$. Then there exists $k$ 
such that  $\alpha_k(j)<m_k-n_k$ for infinitely many $j$s. That is,  the $k$th sequence 
$\{\alpha_k(j)\}_{j=1}^{\infty}$ has a subsequence bounded by $m_k-n_k$. 
We pass to that subsequence  of $\{\Balpha(j)\}$ and still call it $\{\Balpha(j)\}$.  

Next we construct $B$ as follows. Let $k_1\in B_n$ be the smallest integer so that 
$\{\alpha_{k_1}(j)\}_{j=1}^{\infty}$ has a bounded subsequence. Then we pass onto a 
subsequence, still calling it $\{\Balpha(j)\}$, so that $\{\alpha_{k_1}(j)\}_{j=1}^{\infty}$ 
is a constant sequence. If $\alpha_{k}(j)\to \infty$ as $j\to\infty$ for all $k\geq k_1+1$ 
then we stop here and $B=\{k_1\}$. Otherwise, we choose $k_2\geq k_1+1$ to be the 
smallest integer so that  $\{\alpha_{k_2}(j)\}_{j=1}^{\infty}$ has a bounded subsequence. 
Now we pass onto a subsequence, again calling it $\{\Balpha(j)\}$, so that 
$\{\alpha_{k_2}(j)\}_{j=1}^{\infty}$ is a constant sequence. After finitely many steps we 
obtain $B=\{k_1,\ldots,k_p\}$ and a subsequence $\{\Balpha(j)\}$ so that   $\{\alpha_{k}(j)\}$ 
is a constant sequence for every $k\in B$ and $\alpha_{k}(j)\to \infty$ as $j\to\infty$ for 
$k\not\in B$.

Hence for this subsequence of $\{\Balpha(j)\}$, taking limit as $j\to \infty$, we have 
\[\lambda=\lim_{j\to\infty}\lambda_{\mathbf{n},\mathbf{m},\Balpha(j),B_n}
=\lim_{j\to\infty}\prod_{k=1}^{n} \frac{\alpha_k(j)+1}{\alpha_k(j)+n_k+m_k+1}
=\prod_{k\in B}\frac{\alpha_k+1}{\alpha_k+n_k+m_k+1}
=\lambda_{\mathbf{n},\mathbf{m},\Balpha,B}\]
where $\Balpha\in \mathbb{N}_0^n$ such that  $\alpha_k$ is arbitrary for $k\not\in B$ and 
$\alpha_k=\alpha_k(j)$ for $k\in B$. We note that $\alpha_k<m_k-n_k$ for some $k\in B$.

Similarly, if $\lim_{j\to\infty}\lambda_{\mathbf{n},\mathbf{m},\Balpha(j),B_n} =\lambda$ where 

\[\lambda_{\mathbf{n},\mathbf{m},\Balpha(j),B_n}
=\left(\prod_{k=1}^{n}\frac{\alpha_k(j)+1}{\alpha_k(j)+n_k+m_k+1}
- \prod_{k=1}^{n}\frac{(\alpha_k(j)+1)(\alpha_k(j)+n_k-m_k+1)}{(\alpha_k(j)+n_k+1)^2}\right)\]
then 
\[\lambda=\lim_{j\to\infty}\lambda_{\mathbf{n},\mathbf{m},\Balpha(j),B_n} 
=\lambda_{\mathbf{n},\mathbf{m},\Balpha,B},\]
and $\alpha_k\geq  m_k-n_k$ for all $k\in B$. Hence we have shown that 
\begin{align*} 
\sigma(H^*_{\psi}H_{\psi})\subset \{0\}\cup \{\lambda_{\mathbf{n},\mathbf{m},\Balpha,B}: 
\Balpha\in \mathbb{N}_0^n, \emptyset\neq B\subset B_n\}.
\end{align*}

Next we prove the converse of the inclusion above. First, we choose 
$\alpha_k(j)=j$ for all $k$. Then it is easy to see that 
$\lambda_{\mathbf{n},\mathbf{m},\Balpha(j),B_n}\to 0$ as $j\to \infty$. 
Hence $0\in \sigma(H^*_{\psi}H_{\psi})$. 
Secondly, let us assume that  $\emptyset \neq B \subset B_n$ 
and $\alpha_k<m_k-n_k$ for some $k\in B$. We choose $\alpha_k(j)=j$ for $k\not\in B$ 
and $\alpha_k(j)=\alpha_k$ for $k\in B$. Then  
\[\lambda_{\mathbf{n},\mathbf{m},\Balpha(j),B_n} 
\to \lambda_{\mathbf{n},\mathbf{m},\Balpha,B}\in \sigma(H^*_{\psi}H_{\psi}) \text{ as } j\to \infty.\] 
Similarly,  if $\alpha_k\geq m_k-n_k$ for all $k\in B$ we choose $\{\Balpha(j)\}$ as above again  
and conclude that $\lambda_{\mathbf{n},\mathbf{m},\Balpha,B}\in \sigma(H^*_{\psi}H_{\psi})$. 
Therefore, 
\begin{align*} 
\sigma(H^*_{\psi}H_{\psi})=\{0\}\cup \{\lambda_{\mathbf{n},\mathbf{m},\Balpha,B}: 
\Balpha\in \mathbb{N}_0^n, \emptyset\neq B\subset B_n\}.
\end{align*}

We finish the proof by proving the claims about multiplicities. Let us assume that $n_k+m_k\geq 1$ 
for all $k\in B_n$ and $\lambda$ is a non-zero eigenvalue of infinite multiplicity. Then there 
exists a sequence of multi-indices $\{\Balpha(j)\}$ such that  
$\lambda_{\mathbf{n},\mathbf{m},\Balpha(j),B_n}=\lambda$ for all $j$. 
After passing to a subsequence, we may assume that $\lim_{j\to\infty} \alpha_k(j)$ 
exists allowing infinity as the limit for all $k$. 

If $B=\emptyset$ then $\lambda=0$. For the rest of the proof, we assume that $B\neq\emptyset$ 
and, by passing to a subsequence if necessary, $\alpha_k(j)=\alpha_k$ for $k\in B$ and 
$\alpha_k(j)\to\infty$ for $k\not\in B$. In case $\alpha_k(j)< m_k-n_k$ for some $k\in B$, 
it follows that 
\[\lambda_{\mathbf{n},\mathbf{m},\Balpha(j),B_n}
<\prod_{k\in B}\frac{\alpha_k(j)+1}{\alpha_k(j)+n_k+m_k+1}
=\prod_{k\in B}\frac{\alpha_k+1}{\alpha_k+n_k+m_k+1}=\lambda\] 
which is a contradiction with $\lambda_{\mathbf{n},\mathbf{m},\Balpha(j),B_n} =\lambda$ for all $j$. 

On the other hand, if $\alpha_k(j)\geq m_k-n_k$ for all $k=1,2,\dots,n$, the equality 
$\lambda_{\mathbf{n},\mathbf{m},\Balpha(j),B_n}=\lambda$ for all $j$ implies that
\begin{align}\nonumber 
\lambda=\prod_{k\in B}&\frac{\alpha_k(j)+1}{\alpha_k(j)+n_k+m_k+1}
-\prod_{k\in B}\frac{(\alpha_k(j)+1)(\alpha_k(j)+n_k-m_k+1)}{(\alpha_k(j)+n_k+1)^2}\\
\label{Eqn10} &= \prod_{k=1}^n\frac{\alpha_k(j)+1}{\alpha_k(j)+n_k+m_k+1}
-\prod_{k= 1}^n\frac{(\alpha_k(j)+1)(\alpha_k(j)+n_k-m_k+1)}{(\alpha_k(j)+n_k+1)^2}.
\end{align}

Next we will prove that this is impossible for all $j$. Let us set 
\begin{align*}
a&=\prod_{k\in B}\frac{\alpha_k(j)+1}{\alpha_k(j)+n_k+m_k+1},
\\b&=\prod_{k\in B}\frac{(\alpha_k(j)+1)(\alpha_k(j)+n_k-m_k+1)}{(\alpha_k(j)+n_k+1)^2}.
\end{align*} 
Then \eqref{Eqn10} becomes
\begin{align}\label{Eqn12} 
a-b-a\prod_{k\not\in B}\frac{\alpha_k(j)+1}{\alpha_k(j)+n_k+m_k+1}
+b\prod_{k\not\in B}\frac{(\alpha_k(j)+1)(\alpha_k(j)+n_k-m_k+1)}{(\alpha_k(j)+n_k+1)^2}
=0. 
\end{align} 
Rearranging this equation, we obtain a polynomial equation
\begin{align*}
(a-b)&\prod_{k\not\in B}(\alpha_k(j)+n_k+1)^2(\alpha_k(j)+n_k+m_k+1)-a
\prod_{k\not\in B}(\alpha_k(j)+n_k+1)^2(\alpha_k(j)+1)\\
\nonumber &+b\prod_{k\not\in B}(\alpha_k(j)+1)(\alpha_k(j)+n_k-m_k+1)(\alpha_k(j)+n_k+m_k+1)
=0
\end{align*}
Rewriting the equality above, we get 
\begin{align}
\nonumber (a-b)&\prod_{k\not\in B}(\alpha_k(j)+n_k+1)^2 
\left(\prod_{k\not\in B}(\alpha_k(j)+n_k+m_k+1)-\prod_{k\not\in B}(\alpha_k(j)+1)\right)\\
\label{Eqn11}&-b\prod_{k\not\in B}(\alpha_k(j)+1)m^2_k
=0.
\end{align} 
If $m_k\geq 1$ for some $k\in B$ then $a-b>0$ and   
\[\prod_{k\not\in B}(\alpha_k(j)+n_k+m_k+1)-\prod_{k\not\in B}(\alpha_k(j)+1)\geq 1.\] 
Furthermore, $\prod_{k\not\in B}(\alpha_k(j)+n_k+1)^2$ dominates 
$\prod_{k\not\in B}(\alpha_k(j)+1)m^2_k$ as $j\to \infty$. Hence, the left hand side 
of \eqref{Eqn11} converges to $\infty$ as $j\to \infty$, reaching a contradiction. 

On the other hand, if $m_k=0$ for all $k\in B$ and since $\mathbf{m}\neq 0$, we have $m_k\geq 1$ 
for some $k\not\in B$. Then $a-b=0$ and \eqref{Eqn12} implies 
\[-\prod_{k\not\in B}\frac{\alpha_k(j)+1}{\alpha_k(j)+n_k+m_k+1}
+\prod_{k\not\in B}\frac{(\alpha_k(j)+1)(\alpha_k(j)+n_k-m_k+1)}{(\alpha_k(j)+n_k+1)^2}=0.\]
However, the left hand side of the equation above is equal to a positive multiple of the 
following expression
\[-\prod_{k\not\in B} (\alpha_k(j)+n_k+1)^2 
+\prod_{k\not\in B} \left((\alpha_k(j)+n_k+1)^2 -m_k^2\right) \]
which is negative,  reaching a contradiction again. Therefore, we showed that, if $m_k+n_k\geq 1$ 
for all $k\in B_n$ then each eigenvalue is of finite multiplicity. 

Finally, let us define $B_0=\{k\in B_n:m_k=n_k=0\}$. Now we will show that 
all of the eigenvalues have infinite multiplicities when $B_0\neq \emptyset$. We note that since 
$\mathbf{m}\neq 0$ we know that $B_0\subsetneqq  B_n$.

Assume that $\lambda_{\mathbf{n},\mathbf{m},\Balpha,B_n}$ is an eigenvalue. 
Then either $\alpha_k<m_k-n_k$ for some $k$ or $\alpha_k\geq m_k-n_k$ for all $k$. 
In the first case, we have $\alpha_k<m_k-n_k$ for some $k\not\in B_0$. Then 
\[\lambda_{\mathbf{n},\mathbf{m},\Balpha(j),B_n}= 
\prod_{k=1 }^{n}\frac{\alpha_k(j)+1}{\alpha_k(j)+n_k+m_k+1} =
\prod_{k\not\in B_0}\frac{\alpha_k+1}{\alpha_k+n_k+m_k+1}
=\lambda_{\mathbf{n},\mathbf{m},\Balpha,B_n}\] 
for $\alpha_k(j)=j$ for $k\in B_0$ and $\alpha_k(j)=\alpha_k$ for $k\not\in B_0$. 
Similarly, if $\alpha_k\geq m_k-n_k$ for all $k$ we make the same choice 
$\alpha_k(j)=j$ for $k\in B_0$ and $\alpha_k(j)=\alpha_k$ for $k\not\in B_0$ and 
one can see that 
\begin{align*}
\lambda_{\mathbf{n},\mathbf{m},\Balpha(j),B_n}=& \, 
\prod_{k=1}^n\frac{\alpha_k(j)+1}{\alpha_k(j)+n_k+m_k+1}
-\prod_{k= 1}^n\frac{(\alpha_k(j)+1)(\alpha_k(j)+n_k-m_k+1)}{(\alpha_k(j)+n_k+1)^2}\\
=\,&\prod_{k\not\in B_0}\frac{\alpha_k+1}{\alpha_k+n_k+m_k+1}
-\prod_{k\not\in B_0}\frac{(\alpha_k+1)(\alpha_k+n_k-m_k+1)}{(\alpha_k+n_k+1)^2}\\
=\,&\lambda_{\mathbf{n},\mathbf{m},\Balpha,B_n}.
\end{align*} 
Therefore, in either case the eigenvalues have infinite multiplicities.
\end{proof}  

By the proof of Theorem \ref{ThmMonSym}, we can also characterize 
the essential spectrum of Hankel operators with monomial symbols. 
\begin{corollary}
Let $\psi(z)=z^{\mathbf{n}}\zb^{\mathbf{m}}$ for some $\mathbf{m},\mathbf{n}\in \mathbb{N}^n_0$. 
If $m_k+n_k=0$ for some $k\in B_n$, then the essential spectrum of 
$H^*_{\psi}H_{\psi}$ on $A^2(\mathbb{D}^n)$ 
\[\sigma_e(H^*_{\psi}H_{\psi})=\sigma(H^*_{\psi}H_{\psi}) 
=\{0\}\cup \left\{\lambda_{\mathbf{n},\mathbf{m},\Balpha,B}:\Balpha\in \mathbb{N}_0^n, 
\emptyset\neq B\subset B_n\right\}.\]
On the other hand,  if $n_k+m_k\geq 1$ for all $k\in B_n$ and $m_k\geq 1$ for some $k\in B_n$, 
then the essential spectrum of $H^*_{\psi}H_{\psi}$ on $A^2(\mathbb{D}^n)$ 
\[\sigma_e(H^*_{\psi}H_{\psi})=\{0\}\cup \left\{\lambda_{\mathbf{n},\mathbf{m},\Balpha,B}:
\Balpha\in \mathbb{N}_0^n, \emptyset\neq B\subset B_n \text{ and } B\neq  B_n\right\}.\]
\end{corollary}

\begin{proof}
First let us assume that $B_0=\{k\in B_n:m_k+n_k=0\}\neq\emptyset$. Then all of the 
eigenvalues have infinite multiplicity. So the discrete spectrum is empty and the 
essential spectrum is identical to the spectrum. 

Next we assume that $B_0=\emptyset$ and  $m_k\geq 1$ for some $k\in B_n$.  
Then 0 is in the essential spectrum because  
$0=\lim_{j\to\infty}\lambda_{\mathbf{n},\mathbf{m},\Balpha(j),B_n}$ for $\alpha_k(j)=j$ 
for all $k\in B_n$. 
Next  one can see that for any $\Balpha\in \mathbb N^n_0$, the value 
$\lambda_{\mathbf{n},\mathbf{m},\Balpha,B}$ is in the essential spectrum for 
$\emptyset\neq B\subset B_n$ and $B\neq B_n$ as follows. Let $\{\Balpha(j)\}$ be a sequence  
in $\mathbb N^n_0$ with $\alpha_k(j)=\alpha_k$ for $k\in B$ and $\alpha_k(j)=j$ for $k\notin B$. 
Then as shown in the proof of Theorem \ref{ThmMonSym}, the constant 
$\lambda_{\mathbf{n},\mathbf{m},\Balpha(j),B_n}$ is an eigenvalue 
and $\lambda_{\mathbf{n},\mathbf{m},\Balpha,B}
=\lim_{j\to\infty}\lambda_{\mathbf{n},\mathbf{m},\Balpha(j),B_n}$. 

On the other hand, Lemma \ref{LemClosure} implies that every 
$\lambda$ in the essential spectrum is the limit of a sequence of eigenvalues. 
By passing to subsequences argument as in Theorem \ref{ThmMonSym} and the fact 
that $B_0$ is empty, we can see that such a $\lambda$ will be of form 
$\lambda_{\mathbf{n},\mathbf{m},\Balpha,B}$ for some $\Balpha\in \mathbb N^n_0$ 
and some proper subset  $B$ of $B_n$. 
\end{proof}


\end{document}